\documentclass[12pt]{article}%
\usepackage{amsmath}
\usepackage{amsfonts}
\usepackage{amssymb}
\usepackage{graphicx}%
\providecommand{\U}[1]{\protect\rule{.1in}{.1in}}
\newtheorem{theorem}{Theorem}

\newtheorem{definition}[theorem]{Definition}
\newtheorem{example}[theorem]{Example}

\newtheorem{proposition}[theorem]{Proposition}
\newtheorem{remark}[theorem]{Remark}

\newenvironment{proof}[1][Proof]{\noindent\textbf{#1.} }{\ \rule{0.5em}{0.5em}}
\begin{document}

\title{A diagonal principle for nets}
\author{Youssef Azouzi\\{\small Research Laboratory of Algebra, Topology, Arithmetic, and Order}\\{\small Department of Mathematics}\\{\small Faculty of Mathematical, Physical and Natural Sciences of Tunis}\\{\small Tunis-El Manar University, 2092-El Manar, Tunisia}}
\date{}
\maketitle

\begin{abstract}
In our work, we provide a constructive proof of a generalized version of
Cantor's diagonal argument for nets. This result expands the well-known
technique beyond sequences, allowing it to be applied to a broader context.
This result has significant potential applications in various fields, and we
demonstrate a few of these in our work. One such application is the solution
to a previously open problem posed by Kandi\'{c}, Marabeh, and Troitsky.
Specifically, we show that the set of unbounded norm compact operators from a
Banach space to a Banach lattice is closed.

\end{abstract}

\section{Introduction}

The main purpose of this paper is to give a `net version' of diagonal process.
This principle has not been previously observed, and it enables the
generalization several results from sequences to nets. It is worth emphasizing
that our proof is constructive and offers a concrete description of that
desired subnet. Several applications are presented to demonstrate the
practical utility of this method. We show namely that the set of unbounded
norm sequential compact operators from a Banach space $X$ to a Banach lattice
$E$ is norm closed, which answers positively an open question asked in 2017 by
Kandi\'{c}, Marabeh and Troitsky \cite{L-171}. The second one deals with order
compact operators as they were defined in \cite{L-287} and shows a similar
result for this kind of operators.

The diagonal process works as follows: Let us start with a sequence $x=\left(
x_{n}\right)  $ belonging to some space $X$. From which we extract a
subsequence $x^{1}=\left(  x_{\varphi_{1}\left(  n\right)  }\right)  ,$ from
this later we extract a second subsequence $x^{2}=\left(  x_{\varphi_{1}%
\circ\varphi_{2}\left(  n\right)  }\right)  $ and so forth. In the $k^{th}$
step we get a sequence $\left(  x_{\varphi_{1}o....\circ\varphi_{k}\left(
n\right)  }\right)  $ which is a subsequence of the previous one. Now the
point is to consider the "diagonal" sequence $y=\left(  x_{\varphi_{1}%
\circ...\circ\varphi_{n}\left(  n\right)  }\right)  ,$ which possesses the
advantageous property that if we discard some of its initial terms, it can be
treated as a subsequence of each sequence $x^{k}.$ Specifically, $\left(
y_{n}\right)  $ constitutes a shared eventual subsequence of all the sequences
$x^{k},$ $k=1,2,...$. Consequently, if each sequence $x^{k}$ exhibits a
property $\left(  \mathcal{P}_{k}\right)  ,$ that is retained when
transitioning to eventual subsequences, hen $\left(  y_{n}\right)  $ satisfies
all these properties. This idea can be expressed in a suitable manner as follows:

\begin{theorem}
\label{Y19-A}Assume that we have a sequence $\left(  x_{n}\right)  $ in some
space $X$ and a sequence of properties $\left(  \mathcal{P}_{k}\right)  .$ We
assume that the following assumptions are satisfied\newline

\begin{enumerate}
\item[(i)] Every property $\left(  \mathcal{P}_{k}\right)  $ remains true upon
transitioning to eventual subsequences.

\item[(ii)] For every $k\in\mathbb{N}$ and for every subsequence $\left(
y_{n}\right)  $ of $\left(  x_{n}\right)  $ there is a further subsequence
$\left(  y_{\varphi\left(  n\right)  }\right)  $ satisfying the property
$\left(  \mathcal{P}_{k}\right)  .$\newline In this case, there exists a
subsequence of the initial sequence $\left(  x_{n}\right)  $ that satisfies
all properties $\left(  \mathcal{P}_{k}\right)  .$
\end{enumerate}
\end{theorem}

Here $\left(  y_{n}\right)  $ is an eventual subsequence of $\left(
x_{n}\right)  $ if a tail of $\left(  y_{n}\right)  $ is a subsequence of
$\left(  x_{n}\right)  .$

The main purpose of this work is to establish a `net version' of the principle
and subsequently offer a few applications.

\section{Preliminaries}

Let $X$ be an arbitrary nonempty set. A net in $X$ is a map
$x:A\longrightarrow X,$ written $\left(  x_{\alpha}\right)  _{\alpha\in A},$
or simply $\left(  x_{\alpha}\right)  ,$ where $A$ is a directed set. The
later means that there is a binary relation $\leq$ on $A$ satisfying the
following three conditions:

(i) Reflexivity : $\alpha\leq\alpha$ for every $\alpha\in A.$

(ii) Transitivity : $\alpha\leq\beta$ and $\beta\leq\gamma\Longrightarrow
\alpha\leq\gamma.$

(iii) Upward directed : for every $\alpha,\beta\in A$ there exists $\gamma\in
A$ such that $\alpha\leq\gamma$ and $\beta\leq\gamma.$

We write also $\beta\geq\alpha$ instead of $\alpha\leq\beta.$ We recall that
if $X$ is a topological space and $\left(  x_{\alpha}\right)  _{\alpha\in A}$
is a net in $X$ then $\left(  x_{\alpha}\right)  _{\alpha\in A}$ converges to
$x$ if for every neighborhood $V$ of $x$ there exists $\alpha_{0}\in A$ such
that $x_{\alpha}\in V$ for every $\alpha\geq\alpha_{0}.$ We are interested in
this paper to some non topological convergences. Let $L$ be a vector lattice
we say that a net $\left(  x_{\alpha}\right)  _{\alpha\in A}$ is order
convergent to $x$ if there exists another net $\left(  y_{\beta}\right)
_{\beta\in B}$ such that $y_{\beta}\downarrow0$ and for every $\beta\in B,$
there exists $\alpha_{\beta}\in A$ such that $\left\vert x_{\alpha
}-x\right\vert \leq y_{\beta}$ for all $\alpha\geq\alpha_{0}.$ A net $\left(
x_{\alpha}\right)  _{\alpha\in A}$ is said to be unbounded order convergent to
$x$ if for every $u\in L^{+},$ the positive cone of $L,$ the net $\left\vert
x_{\alpha}-x\right\vert \wedge u$ is order convergent to $0.$ This kind of
convergence is interesting because for spaces $L^{p}$ it coincides for
sequences with almost everywhere convergence and for spaces $\ell^{p}$ it
agrees with pointwise convergence. Another kind of unbounded convergence is
unbounded norm convergence studied for example in \cite{L-171} and
\cite{L-173}. Recall that if $X$ is a Banach lattice then a net $\left(
x_{\alpha}\right)  $ is said to be unbounded norm convergent to $x$ and we
write $x_{\alpha}\overset{\mathfrak{un}}{\longrightarrow}x$ if $\left\vert
x_{\alpha}-x\right\vert \wedge u$ converges in norm to zero for every $u\in
X^{+}.$ In $L^{p}\left(  \mu\right)  $ we have $f_{n}\overset{\mathfrak{un}%
}{\longrightarrow}f$ if and only if $f_{n}\longrightarrow f$ in measure.

We need also the following definition.

\begin{definition}
We say that $\left(  y_{\beta}\right)  $ is an eventual subnet (or is
eventually a subnet) of $\left(  x_{\alpha}\right)  $ if there is $\beta_{0}$
such that $\left(  y_{\beta}\right)  _{\beta\geq\beta_{0}}$ is a subnet of
$\left(  x_{\alpha}\right)  .$
\end{definition}

\section{The diagonal process for nets}

The diagonal argument is a powerful technique commonly employed to establish
that if a sequence of maps, denoted by $\left(  f_{n}\right)  ,$ satisfies a
given property $\left(  \mathcal{P}\right)  $ and converges to a map $f$, then
$f$ possesses that same property$.$ One typical application of this method is
demonstrating that the collection of compact operators defined on a Banach
space $X$ is a closed subspace of $L\left(  X\right)  ,$ the space of all
bounded operators from $X$ to its self.

Interestingly, it is often observed in literature that the diagonal argument
for sequences can provide a constructive proof for the positive resolution of
problems in certain situations. However, in cases where each map $f_{n}$
satisfies a property $\left(  \ast\right)  $ related to nets, the problem
becomes more challenging to solve, as there is currently no equivalent
diagonal process for nets available. To illustrate this point, two open
questions from recent papers \cite{L-171} by Kandi\'{c}, Marabeh, and
Troitsky, and \cite{L-287} by Ayd\i n, Emelyanov, Erkur\c{s}un \"{O}zcan and
Marabeh, are worth mentioning as examples of this situation.

In general, properties $\left(  \mathcal{P}_{k}\right)  $ are closely
connected to convergence. Therefore, the following case, which follows as a
consequence of Theorem \ref{Y19-A}, may be with a particular interest.

\begin{theorem}
\label{Y19-B}Let $\left(  E_{k}\right)  $ be a sequence of convergence spaces
(as in \cite{L-748}) and $\left(  x_{n}\right)  $ be a sequence in the product
space $P=%
{\textstyle\prod\limits_{k=1}^{\infty}}
E_{k}.$ We assume that for every subsequence $\left(  y_{n}\right)  $ of
$\left(  x_{n}\right)  $ and for every integer $k,$ there is a subsequence
$\left(  y_{\varphi\left(  n\right)  }\right)  $ such that $y_{\varphi\left(
n\right)  }\left(  k\right)  $ converges in $E_{k}.$ Then there is a
subsequence $\left(  z_{n}\right)  $ of $\left(  x_{n}\right)  $ such that
\[
z_{n}\left(  k\right)  \text{ converges in }E_{k}\text{ for }k=1,2,....
\]

\end{theorem}

\begin{remark}
It should be noted that the assumption `eventual subsequences' in Theorem
\ref{Y19-A}.(i) could not be replaced by `subsequences'. To illustrate this,
consider the real sequence $x_{n}=\dfrac{1}{\sqrt{n}}$ and let us say that a
sequence satisfies property $\left(  \mathcal{P}_{k}\right)  $ if $\left\vert
y_{n}\right\vert \leq\dfrac{1}{k^{n}}$ for $n=1,2,....$ It is clear that
$\left(  \mathcal{P}_{k}\right)  $ is preserved under passing to subsequences
and that each subsequence of $\left(  x_{n}\right)  $ has a further
subsequence satisfying $\left(  \mathcal{P}_{k}\right)  .$ However, there is
no subsequence of $\left(  x_{n}\right)  $ satisfying all properties $\left(
\mathcal{P}_{k}\right)  $ because such sequence would be null, and $x_{n}>0$
for all $n.$
\end{remark}

In the next simple example we illustrate Theorem \ref{Y19-A} by considering
some situation when properties $\left(  \mathcal{P}_{k}\right)  $ are not
necessarily related to convergence.

\begin{example}
It is well known that a metric space is compact if and only if is totally
bounded and complete. We will show that if $X$ is totally bounded and complete
then $X$ is sequentially compact, meaning that every sequence in $X$ has a
cluster point. To this end let us fix a sequence $\left(  x_{n}\right)  $ in
$X$ and consider for each integer $k$ the property $\mathcal{P}_{k}$ : A
sequence $\left(  y_{n}\right)  $ has the property $\mathcal{P}_{k}$ if it has
a tail contained in a ball with radius $\dfrac{1}{k}.$ Since $X$ is totally
bounded it is quite obvious that every sequence in $X$ satisfies
$\mathcal{P}_{k}.$ It follows from Theorem \ref{Y19-A} that $\left(
x_{n}\right)  $ has a subsequence $\left(  z_{n}\right)  $ satisfying all
properties $\left(  \mathcal{P}_{k}\right)  .$ In particular $\left(
z_{n}\right)  $ is a Cauchy sequence and so it converges. Thus $\left(
x_{n}\right)  $ has a cluster point as required.
\end{example}

The primary objective of this paper is to extend the previous results to nets.
However, generalizing results to nets is challenging due to the index set
changing when a new subnet is considered. The main result of our study can be
expressed as follows:

\begin{theorem}
\label{Y19-C}Consider a sequence $\left(  E_{k}\right)  $ of convergence
spaces and let $\left(  x_{\alpha}\right)  _{\alpha\in A}$ be a net in the
product space $P=%
{\textstyle\prod\limits_{k=1}^{\infty}}
E_{k}.$ We assume that for every subnet $\left(  y_{\beta}\right)  _{\beta\in
B}$ of $\left(  x_{\alpha}\right)  $ and for every integer $k,$ there is a
subnet $\left(  y_{\varphi\left(  \gamma\right)  }\right)  _{\gamma\in C}$
such that $y_{\varphi\left(  \gamma\right)  }\left(  k\right)  $ converges in
$E_{k}.$ Then there is a subnet $\left(  z_{\theta}\right)  $ of $\left(
x_{\alpha}\right)  $ such that
\[
z_{\theta}\left(  k\right)  \text{ converges in }E_{k}\text{ for }k=1,2,....
\]

\end{theorem}

Our result can be expressed in a more general form as follows.

\begin{theorem}
\label{Y19-D}Let $E$ be a nonempty set, $\left(  x_{\alpha}\right)
_{\alpha\in A}$ a net in $E$ and $\left(  \mathcal{P}_{k}\right)
_{k\in\mathbb{N}}$ a sequence of properties that are preserved under passing
to eventual subnets. Additionally, assume that for each $k$ and for each
subnet $\left(  y_{\beta}\right)  _{\beta\in B}$ of $\left(  x_{\alpha
}\right)  $ there is a further subnet $\left(  y_{\varphi\left(
\gamma\right)  }\right)  _{\gamma\in C}$ such that $y_{\varphi\left(
\gamma\right)  }$ satisfies $\mathcal{P}_{k}.$ Then there is a subnet $\left(
z_{\theta}\right)  $ of the initial net $\left(  x_{\alpha}\right)  $ such
that $\left(  z_{\theta}\right)  $ satisfies $\mathcal{P}_{k}$ for all
$k=1,2,....$
\end{theorem}

\begin{proof}
It follows immediately from our assumptions that there is a sequence
$x^{n}=\left(  x_{\alpha_{n}}^{n}\right)  _{\alpha_{n}\in A_{n}}$ of subnets
of $\left(  x_{\alpha}\right)  _{\alpha\in A}$ such that $x^{n}$ satisfies
$\mathcal{P}_{n}$ and $x^{n}$ is a subnet of $x^{n-1}$ for each $n,$ with
$A_{0}=A$ and $x^{0}=\left(  x_{\alpha}\right)  _{\alpha\in A_{0}}.$ This is
achieved by using cofinal maps $\varphi_{n}:A_{n}\longrightarrow A_{n-1},$
$x_{\alpha_{n}}^{n}=x_{\varphi_{n}\left(  \alpha_{n}\right)  }^{n-1},$
$n=1,2,...$ As every property $\mathcal{P}_{k}$ is preserved under passing to
eventual subnets, it follows that for each $n,$ the subnet $x^{n}$ satisfies
properties $\mathcal{P}_{k}$ for $k=1,2,....,n.$ Next, we will construct a net
$\left(  y_{\beta}\right)  _{\beta\in B}$ which is a subnet of our initial net
and an eventual subnet of each $x^{n}$ so that it satisfies all properties
$\left(  \mathcal{P}_{n}\right)  .$ Its index set $B$ is consisting of all
elements $\beta$ of the form%
\[
\beta=\left(  \beta_{0},\beta_{1},...,\beta_{k}\right)  \in A_{0}\times
A_{1}\times...\times A_{k},
\]
where $k$ is an integer, and for each $i\in\left\{  1,2,...,k\right\}  ,$ we
have $\beta_{i-1}=\varphi_{i}\left(  \beta_{i}\right)  .$ We order the set $B$
by putting :%
\begin{align*}
\left(  \beta_{0},\beta_{1},...,\beta_{p}\right)   &  \preccurlyeq\left(
\gamma_{0},\gamma_{1},...,\gamma_{q}\right) \\
&  \Leftrightarrow p\leq q\text{ and }\beta_{i}\leq\gamma_{i}\text{ for }0\leq
i\leq p.
\end{align*}

(i) We claim that the ordered set $\left(  B,\preccurlyeq\right)  $ is
directed. To this end consider tow elements $\beta=\left(  \beta_{0},\beta
_{1},...,\beta_{p}\right)  $ and $\gamma=\left(  \gamma_{0},\gamma
_{1},...,\gamma_{q}\right)  $ in $B,$ and assume, for instance, that $p\leq
q$. As $A_{0}=A$ is directed there is $\theta_{0}\in A_{0}$ satisfying
$\theta_{0}\geq\beta_{0}$ and $\theta_{0}\geq\gamma_{0}.$ As $\varphi_{1}$ is
cofinal there exists $\alpha_{1}^{\prime}\in A_{1}$ such that $\alpha_{1}%
\geq\alpha_{1}^{\prime}\Longrightarrow\varphi_{1}\left(  \alpha_{1}\right)
\geq\theta_{0}.$ We choose $\theta_{1}\in A_{1}$ such that $\theta_{1}%
\geq\alpha_{1}^{\prime},\beta_{1},\gamma_{1}.$ Now we choose $\beta
_{p+1},...,\beta_{q}$ arbitrary in $A_{p+1},...,A_{q}$ and by iterating the
previous process we can construct a sequence $\theta_{i}\in A_{i},$ $1\leq
i\leq q$ such that $\theta_{i}\geq\beta_{i},$ $\theta_{i}\geq\gamma_{i},$ and%
\[
\theta\geq\theta_{i}\Longrightarrow\varphi_{i}\left(  \theta\right)
\geq\theta_{i-1}.
\]
Define $\alpha=\left(  \alpha_{0},...,\alpha_{q}\right)  $ by putting
$\alpha_{q}=\theta_{q}$ and for $i=1,...,q,$ $\varphi_{i}\left(  \alpha
_{i}\right)  =\alpha_{i-1}.$ It is quite clear that $\alpha$ belongs to $B$
and satisfies $\beta\preccurlyeq\alpha$ and $\gamma\preccurlyeq\alpha.$ This
shows that $\left(  B,\preccurlyeq\right)  $ is directed. We can then define
the desired subnet by setting $y_{\beta}=x_{\beta_{0}}$ for $\beta=\left(
\beta_{0},...,\beta_{n}\right)  \in B.$

(ii) We need to show that the net $y=\left(  y_{\beta}\right)  _{\beta\in B}$
is a subnet of $\left(  x_{\alpha}\right)  .$ To do so, it suffices to show
that the map $\varphi:B\longrightarrow A$ defined by $\varphi\left(
\beta\right)  =\beta_{0},$ is cofinal. Let $\alpha_{0}\in A$ be fixed. We aim
to find $\gamma\in B$ such that $\varphi\left(  \beta\right)  \geq\alpha_{0}$
whenever $\beta\geq\gamma.$ As $\varphi_{1}$ is cofinal, there exists
$\gamma_{1}\in A_{1}$ such that%
\[
\alpha_{1}\geq\gamma_{1}\Longrightarrow\varphi_{1}\left(  \alpha_{1}\right)
\geq\alpha_{0}.
\]
It follows that for each $\beta=\left(  \beta_{0},...,\beta_{p}\right)  \in B$
satisfying $\beta\succcurlyeq\gamma=\left(  \varphi\left(  \gamma_{1}\right)
,\gamma_{1}\right)  $ we have $\beta_{1}\geq\gamma_{1}$ and so $\varphi\left(
\gamma\right)  =\varphi_{1}\left(  \beta_{1}\right)  \geq\alpha_{0}.$ Thus
$\varphi$ is cofinal as it was claimed.

(iii) To complete our proof we need to show that $y$ is an eventual subnet of
$x^{n}$ for each $n.$ To this end, let us fix an element $\beta^{\prime
}=\left(  \beta_{0}^{\prime},\beta_{1}^{\prime}...,\beta_{n+1}^{\prime
}\right)  $ in $B.$ We will prove that the net $\left(  y_{\beta}\right)
_{\beta\geq\beta^{\prime}}$ is a subnet of $x^{n}.$ A cofinal map $\psi_{n}$
from $\Gamma_{n}:=\left\{  \beta\in B:\beta\succcurlyeq\beta^{\prime}\right\}
$ to $A_{n}$ should be defined. Our suitable candidate will be given by:%
\[
\varphi_{n}\left(  \beta\right)  =\beta_{n}\text{ for }\beta=\left(  \beta
_{1},....,\beta_{q}\right)  \in\Gamma_{n}.
\]
Let $\alpha_{n}$ be a fixed element in $A_{n}.$ Choose $\gamma_{n+1}\in
A_{n+1}$ such that%
\[
\alpha_{n+1}\geq\gamma_{n+1}\Longrightarrow\varphi_{n+1}\left(  \alpha
_{n+1}\right)  \geq\alpha_{n}.
\]
Now let $\gamma=\left(  \gamma_{0},...,\gamma_{n+1}\right)  $ with
$\gamma_{i-1}=\varphi_{i}\left(  \gamma_{i}\right)  $ for $i=1,...,n+1,$ and
let $\theta\in B$ satisfying $\theta\succcurlyeq\beta^{\prime}$ and
$\theta\succcurlyeq\gamma.$ Then for every $\beta\in B$ we have%
\[
\beta\succcurlyeq\theta\Longrightarrow\beta_{n}=\psi_{n}\left(  \beta\right)
\geq\theta_{n}\geq\alpha_{n}.
\]
This shows that $\psi_{n}$ is cofinal and completes the proof.
\end{proof}

\section{Examples of application}

We present in this section a few results as application of our main Theorem.
The two first applications are just new short proofs of known results in
topology and functional analysis. The two next applications are new results
and concern compact operators, answering open questions raised recently in
\cite{L-173} and \cite{L-287}, where authors treated the sequel case following
classical diagonal processes for sequences and left open the general case
because of the lack of a similar result, such result is provided by Theorem
\ref{Y19-D}. It is natural to ask if there is a more general diagonal process
where the set of properties is not assumed countable. In that case a
constructive proof can not be expected, and a use of Zorn's lemma seems to be
evident. This kind of result could be applied in more general situations and
allows to get a very short proofs for Tychonoff Theorem, Alaoglu theorem and
several other results in topology and functional analysis. We plan to explore
this idea in future work.

As stated earlier, we present a short proof of Tychonoff's Theorem for the
countable product of compact spaces. While this proof is well-known for
metrizable spaces, it has not been observed before for general topological
compact spaces.

\begin{theorem}
A countable product of compact spaces is compact.
\end{theorem}

\begin{proof}
Let $\left(  X_{n}\right)  $ be a sequence of compact spaces, $X=%
{\textstyle\prod\limits_{n=1}^{\infty}}
X_{n}$ their product. It is a direct application of Theorem \ref{Y19-C} that
every $\left(  x_{\alpha}\right)  $ in $X$ has a converging subnet, which
shows that $X$ is compact.
\end{proof}

Note that the results we have presented can also be applied to compactness in
the sense of convergence structures, which may not necessarily be topological.
As an example we say that a subset $A$ in a vector lattice $X$ is said to be
order compact if every net in $A$ has a convergent subnet with limit in $A.$
Theorem \ref{Y19-B} says that if $A_{k}$ is order compact in the vector
lattice $V_{k}$ then their product $A=\prod\limits_{k=1}^{\infty}A_{k}$ is
order compact in the vector lattice $V=\prod\limits_{k=1}^{\infty}V_{k}.$

By applying our main theorem to the dual space of a separable Banach space
$X$, we can obtain a short proof of the Banach-Alaoglu Theorem.

\begin{proposition}
If a Banach space $X$ is separable then the unit ball $B_{X^{\ast}}$ of
$X^{\ast}$ is weakly* compact.
\end{proposition}

\begin{proof}
Let $D$ be a dense countable subset of $X$ and let $\left(  x_{\alpha}^{\ast
}\right)  $ be a net in $B_{X^{\ast}}.$ By Theorem \ref{Y19-A} there is a
subnet $\left(  y_{\beta}^{\ast}\right)  $ of $\left(  x_{\alpha}\right)  $
such that $y_{\beta}^{\ast}\left(  x_{n}\right)  $ converges for every $x\in
D.$ It follows easily that $y_{\beta}^{\ast}\left(  x\right)  $ converges for
every $x\in X.$ Denote $y^{\ast}\left(  x\right)  =\lim y_{\beta}^{\ast
}\left(  x\right)  .$ It is easy to see that $y$ is linear and as $\left\vert
y_{\beta}^{\ast}\left(  x\right)  \right\vert \leq\left\Vert x\right\Vert $
for all $\beta\in B$ we get $\left\vert y^{\ast}\left(  x\right)  \right\vert
\leq\left\Vert x\right\Vert ,$ so $y^{\ast}\in B_{X^{\ast}},$ which completes
the proof.
\end{proof}

\section{Two applications}

Let $E$ be a Banach space and $X$ a Banach lattice and denote by $L\left(
E,X\right)  $ the space of all linear continuous operators from $E$ to $X.$ An
operator $T:E\longrightarrow X$ is said to be sequentially un-compact if every
bounded sequence $\left(  x_{n}\right)  $ has a subsequence $\left(
x_{\varphi\left(  n\right)  }\right)  $ such that $Tx_{\varphi\left(
n\right)  }$ is un-convergent. On the other hand we say that $T$ is un-compact
if every bounded net $\left(  x_{\alpha}\right)  $ in $X$ has a subnet
$\left(  y_{\beta}\right)  $ such that $\left(  Ty_{\beta}\right)  $ is
un-convergent. It was previously shown in \cite[Proposition 9.2]{L-171} that
the set of all sequentially un-compact operators in $L(E,X)$ is a closed
subspace of $L(E,X)$, but it remained an open question whenever the limit of a
sequence of un-compact operators in $L\left(  X,E\right)  $ was still
un-compact. In this work we provide a positive answer to this question.

\begin{theorem}
Let $E$ be a Banach space and $X$ a Banach lattice. Then the set
$K_{un}\left(  E,X\right)  $ of all un-compact operators from $E$ to $X$ is a
closed subspace of $L\left(  E,X\right)  .$
\end{theorem}

\begin{proof}
The fact that $K_{un}\left(  E,X\right)  $ is a subspace of $L\left(
E,X\right)  $ is obvious. We need only to show the closedness. Let $\left(
T_{n}\right)  $ be a sequence of un-compact operators in $L\left(  E,X\right)
$ that converges to $T.$ Let $\left(  x_{\alpha}\right)  $ be a norm bounded
net in $E.$ We aim to show that $\left(  Tx_{\alpha}\right)  $ has a
un-convergent subnet. We say that a net $\left(  u_{i}\right)  $ in $X$
satisfies property $\left(  \mathcal{P}_{n}\right)  $ if $T_{n}u_{i}$ is
un-convergent in $X.$ According to our assumption each property $\mathcal{P}%
_{k}$ is perseved when passing to eventual subnets. As $T_{n}$ is un-compact
each subnet $\left(  y_{\beta}\right)  $ of $\left(  x_{\alpha}\right)  $ has
a further subnet satisfying property $\mathcal{P}_{k}.$ Therefore by applying
Theorem \ref{Y19-D} there is a subnet $\left(  y_{\beta}\right)  $ of $\left(
x_{\alpha}\right)  $ such that $T_{n}y_{\beta}\overset{un}{\longrightarrow
}z_{n}$ for $n=1,2,...$ In order to conclude our proof, we need to establish
the following facts.
\end{proof}

(i) $z_{n}\overset{un}{\longrightarrow}z$ for some $z\in X.$

(ii) $Ty_{\beta}\overset{un}{\longrightarrow}z.$

\textbf{Proof of (i).} For any $u\in X_{+},$ we have the inequality%

\[
\left\Vert \left\vert z_{p}-z_{q}\right\vert \wedge u\right\Vert
\leq\left\Vert \left\vert z_{p}-T_{p}y_{\beta}\right\vert \wedge u\right\Vert
+\left\Vert \left\vert T_{q}y_{\beta}-z_{q}\right\vert \wedge u\right\Vert
+\left\Vert \left\vert T_{p}y_{\beta}-T_{q}y_{\beta}\right\vert \wedge
u\right\Vert .
\]
Taking the $\lim\sup$ over $\beta$ we get%

\[
\left\Vert \left\vert z_{p}-z_{q}\right\vert \wedge u\right\Vert \leq
\limsup\limits_{\beta}\left\Vert \left\vert \left(  T_{q}-T_{p}\right)
y_{\beta}\right\vert \wedge u\right\Vert \leq\left\Vert T_{p}-T_{q}\right\Vert
\sup\left\Vert y_{\beta}\right\Vert .
\]
This gives, for $u=\left\vert z_{p}\right\vert +\left\vert z_{q}\right\vert ,$%

\[
\left\Vert z_{p}-z_{q}\right\Vert \leq\left\Vert T_{p}-T_{q}\right\Vert
\sup\left\Vert y_{\beta}\right\Vert .
\]
This proves that $\left(  z_{n}\right)  $ is a Cauchy sequence and establishes
(i) as $X$ is a Banach space.

\textbf{Proof of (ii) }Let $u\in X_{+}$ and let $M=\sup\left\Vert y_{\beta
}\right\Vert .$ Then we have for all $n,$%

\begin{align*}
\left\Vert \left\vert Ty_{\beta}-z\right\vert \wedge u\right\Vert  &
\leq\left\Vert \left\vert Ty_{\beta}-T_{n}y_{\beta}\right\vert \wedge
u\right\Vert +\left\Vert \left\vert T_{n}y_{\beta}-z_{n}\right\vert \wedge
u\right\Vert +\left\Vert \left\vert z_{n}-z\right\vert \wedge u\right\Vert \\
&  \leq\left\Vert T-T_{n}\right\Vert M+\left\Vert \left\vert T_{n}y_{\beta
}-z_{n}\right\vert \wedge u\right\Vert +\left\Vert z_{n}-z\right\Vert .
\end{align*}
Now taking the $\lim\sup$ over $\beta$ we get%

\[
\limsup\limits_{\beta}\left\Vert \left\vert Ty_{\beta}-z\right\vert \wedge
u\right\Vert \leq\left\Vert T-T_{n}\right\Vert M+\left\Vert z_{n}-z\right\Vert
.
\]

Now (ii) follows from (i) and the fact that $T_{n}\longrightarrow T.$ This
completes the proof.

In \cite{L-287}, several types of compact operators were introduced in . Let
$\left(  X,p,E\right)  $ be a decomposable lattice normed space and $\left(
Y,q,F\right)  $ be a lattice normed space with $F$ is Dedekind complete. We
know that every dominated operator $T:X\longrightarrow Y$ has an exact
dominant $\left\vert T\right\vert $ (see ??). Thus, the triple $\left(
M\left(  X,Y\right)  ,\pi,L^{\sim}\left(  X,Y\right)  \right)  $ is a lattice
normed space with $\pi\left(  T\right)  =\left\vert T\right\vert $ (see,...Ref
in \cite{L-287}). Thus if $\left(  T_{\alpha}\right)  $ is a net in $M\left(
X,Y\right)  $ then $T_{\alpha}\overset{\mathrm{p}}{\longrightarrow}T$ in
$M\left(  X,Y\right)  $ if and only if $\left\vert T_{\alpha}-T\right\vert
\overset{o}{\longrightarrow}0$ in $L^{\sim}\left(  X,Y\right)  .$ It is proved
in \cite[Theorem 5]{L-287} that under these assumptions if $\left(
T_{n}\right)  $ is a sequence of sequentially $p$-compact operators that
$p$-converges in $M\left(  X,Y\right)  $ to $T$ then $T$ is sequentially
$p$-compact. Again here there is no version for nets.

\begin{theorem}
Let $E$ and $F$ be two Riesz spaces with $F$ is Dedekind complete. Let
$\left(  T_{n}\right)  $ be a sequence of order compact operators from $E$ to
$F.$ We assume that $T_{n}\longrightarrow T$ relatively uniformly. Then $T$ is
order compact.
\end{theorem}

\begin{proof}
It follows from our assumptions that there exists an operator $S\in
L_{b}\left(  E,F\right)  $ and a real sequence $\varepsilon_{n}$, decreasing
to $0,$ satisfying the inequality $\left\vert T_{n}-T\right\vert
\leq\varepsilon_{n}S.$ We shall say that a net $\left(  u_{\alpha}\right)  $
satisfies property $\mathcal{P}_{n}$ if the net $\left(  T_{n}u_{\alpha
}\right)  $ is ru-convergent in $F.$ Consider now an order bounded net
$\left(  x_{\alpha}\right)  _{\alpha\in A}$ in $E.$ The application of Theorem
\ref{Y19-A} yields a subnet $\left(  z_{\beta}\right)  $ of $\left(
x_{\alpha}\right)  $ and a sequence $\left(  y_{n}\right)  $ in $F$ such that%
\[
T_{n}x_{\varphi\left(  w\right)  }\overset{o}{\longrightarrow}y_{n}\text{ for
all }n=1,2,...
\]
We claim now that%
\begin{equation}
y_{n}\overset{ru}{\longrightarrow}y.\label{C1}%
\end{equation}
Since $F$ is Dedekind complete there is a net $\left(  z_{\alpha_{n}}\right)
_{\alpha_{n}\in A_{n}}$ decreasing to zero such that%
\[
\left\vert Tx_{\psi_{n}\left(  \alpha_{n}\right)  }-y_{n}\right\vert \leq
z_{\alpha_{n}}.
\]
Now we have%
\begin{align*}
\left\vert y_{p}-y_{q}\right\vert  &  \leq\left\vert y_{p}-T_{p}x_{\psi
_{p}\left(  \alpha_{p}\right)  }\right\vert +\left\vert T_{p}x_{\psi
_{p}\left(  \alpha_{p}\right)  }-T_{q}x_{\psi_{q}\left(  \alpha_{q}\right)
}\right\vert +\left\vert T_{q}x_{\psi_{q}\left(  \alpha_{q}\right)  }%
-x_{q}\right\vert \\
&  \leq\gamma_{\alpha_{p}}+\left\vert T_{p}-T_{q}\right\vert \left\vert
u\right\vert +\gamma_{\alpha_{q}}.
\end{align*}
Taking the limit over $\alpha_{p}$ and $\alpha_{q}$ we get:
\[
\left\vert y_{p}-y_{q}\right\vert \leq\left\vert T_{p}-T_{q}\right\vert
\left\vert u\right\vert .
\]
Now we prove the second claim. We have for every $n,$%
\begin{align*}
\left\vert Tx_{\varphi\left(  w\right)  }-y\right\vert  &  \leq\left\vert
Tx_{\varphi\left(  w\right)  }-T_{n}x_{\varphi\left(  w\right)  }\right\vert
\\
&  +\left\vert T_{n}x_{\varphi\left(  w\right)  }-y_{n}\right\vert +\left\vert
y_{n}-y\right\vert \\
&  \leq\left\vert T_{n}-T\right\vert \left\vert u\right\vert +\left\vert
y_{n}-y_{n}\right\vert +\left\vert T_{n}x_{\varphi\left(  w\right)  }%
-y_{n}\right\vert .
\end{align*}
It follows from (\ref{C1}) that%
\[
\lim\sup\left\vert Tx_{\varphi\left(  w\right)  }-y\right\vert \leq\left\vert
T_{n}-T\right\vert \left\vert u\right\vert +\left\vert y_{n}-y_{n}\right\vert
.
\]
But since $y_{n}\longrightarrow y$ and $T_{n}\longrightarrow T$ we obtain
$\lim\sup\left\vert Tx_{\varphi\left(  w\right)  }-y\right\vert =0,$ which
shows that $Tx_{\varphi\left(  w\right)  }\overset{o}{\longrightarrow}y$ as required.
\end{proof}


\begin{thebibliography}{9}                                                                                                %
\bibitem {L-287}A. Ayd\i n, E.Yu. Emelyanov, N. Erkur\c{s}un \"{O}zcan, M.A.A.
Marabeh, Indag. Math. (N.S.) 29 (2018) 633--656.

\bibitem {L-171}M. Kandi\'{c}, M. Marabeh, V.G. Troitsky, Unbounded norm
topology in Banach lattices, J. Math. Anal. Appl. 451 (1) (2017) 259--279.

\bibitem {L-173}Y. Deng, M O'BrienV.G. Troitsky, Unbounded norm convergence in
Banach lattices, Positivity 21 (3) (2017) 963--974.

\bibitem {L-65}N. Gao, V. Troitsky, F. Xanthos, Uo-convergence and its
applications to Ces\`{a}ro means in Banach lattices, Israel J. Math. 220
(2017) 649--689, https://doi.org/10.1007/s11856-017-1530-y.

\bibitem {b-2474}A.G. Kusraev, Dominated Operators, in: Mathematics and its
Applications, vol. 519, Kluwer Academic Publishers, Dordrecht, 2000, p. xiv+446.

\bibitem {L-748}M. O'Brien, V.G. Troitsky \& J.H. van der Walt (2022): Net
convergence structures with applications to vector lattices, Quaestiones
Mathematicae, 1--38, 2022.
\end{thebibliography}
\end{document}